\documentclass[10pt]{amsart}

\usepackage[colorlinks]{hyperref}
\usepackage{color,graphicx,shortvrb}
\usepackage[latin 1]{inputenc}

\usepackage[active]{srcltx} 

\usepackage{enumerate}
\usepackage{amssymb}

\newtheorem{theorem}{Theorem}[section]

\newtheorem{corollary}[theorem]{Corollary}

\newtheorem{lemma}[theorem]{Lemma}

\newtheorem{proposition}[theorem]{Proposition}
\newtheorem{remark}[theorem]{Remark}

\def\11{\textbf{$1$}}

\DeclareGraphicsExtensions{.jpg,.pdf,.png,.eps}

\begin{document}

\title[The Mazur-Ulam property for the space of complex null sequences]{The Mazur-Ulam property for the space of complex null sequences}

\author[A. Jim{\'e}nez-Vargas]{Antonio Jim{\'e}nez-Vargas}
\author[A. Morales]{Antonio Morales Campoy}
\address[A. Jim{\'e}nez-Vargas, A. Morales, M.I. Ram\'{i}rez]{Departamento de Matem{\'a}ticas, Universidad de Almer{\'i}a, 04120,
Almer{\i}a, Spain}
\email{ajimenez@ual.es, amorales@ual.es, mramirez@ual.es}

\author[A.M. Peralta]{Antonio M. Peralta}

\address[A.M. Peralta]{Departamento de An{\'a}lisis Matem{\'a}tico, Facultad de
Ciencias, Universidad de Granada, 18071 Granada, Spain.}
\email{aperalta@ugr.es}

\author[M.I. Ram\'{i}rez]{Mar{\'\i}a Isabel Ram{\'\i}rez}




\date{}

\begin{abstract} Given an infinite set $\Gamma$, we prove that the space of complex null sequences $c_0(\Gamma)$ satisfies the Mazur-Ulam property, that is, for each Banach space $X$, every surjective isometry from the unit sphere of $c_0(\Gamma)$ onto the unit sphere of $X$ admits a (unique) extension to a surjective real linear isometry from $c_0(\Gamma)$ to $X$. We also prove that the same conclusion holds for the finite dimensional space $\ell_{\infty}^m$.
\end{abstract}

\maketitle
\thispagestyle{empty}

\section{Introduction}

The {Mazur-Ulam property} is intrinsically linked to the so-called \emph{Tingley's problem}. The latter problem has been intensively studied during the last thirty years, and asks whether a surjective isometry $\Delta$ between the unit spheres, $S(X)$ and $S(Y)$, of two normed spaces $X$ and $Y$ can be extended to a surjective real linear isometry from $X$ to $Y$. Tingley's problem has been reveled as a difficult problem which remains unsolved. A wide list of positive solutions to Tingley's problem for concrete Banach spaces includes sequence spaces (see \cite{Ding2002, Di:p, Di:C, Di:8, Di:1, Ding07,LiRen2008, Liu2007}), finite dimensional polyhedral spaces \cite{KadMar2012}, $C_0(L)$ spaces \cite{Wang}, finite dimensional C$^*$-algebras and finite von Neumann algebras (see \cite{Tan2016-2, Tan2016, Tan2016preprint}), spaces of compact operators, compact C$^*$-algebras and weakly compact JB$^*$-triples (cf. \cite{PeTan16,FerPe17}), type I von Neumann factors, atomic von Neumann algebras and atomic JBW$^*$-triples (see \cite{FerPe17b, FerPe17c}), and spaces of trace class operators \cite{FerGarPeVill17}.\smallskip

In general, a surjective linear isometry between the unit spheres of two complex Banach spaces need not admit an extension to a surjective complex linear or conjugate linear isometry between the spaces. For example $\Delta : S( \mathbb{C}\oplus^{\infty}\mathbb{C})\to S( \mathbb{C}\oplus^{\infty}\mathbb{C})$, $\Delta (\lambda_1,\lambda_2) := (\lambda_1,\overline{\lambda_2})$ cannot be extended to a complex linear nor to a conjugate linear isometry on $\mathbb{C}\oplus^{\infty}\mathbb{C}$. Due to these reasons most of the studies are restricted to real Banach spaces, real sequence spaces, and spaces of real-valued measurable functions. However, the results for $C(K)$ spaces and the recent progress for spaces of compact operators, $B(H)$ spaces and atomic von Neumann algebras reveal the importance, validity, and difficulty of the case of complex Banach spaces.\smallskip

Following \cite{ChenDong2011}, we shall say that a Banach space $Z$ satisfies the \emph{Mazur-Ulam property} if every surjective isometry from the unit sphere of $Z$ to the unit sphere of any Banach space $Y$ admits a (unique) extension to a surjective real linear isometry from $Z$ onto $Y$. A pioneering contribution due to G.G. Ding proves that the space $c_0(\mathbb{N},{\mathbb{R}})$ of all null sequences of real numbers satisfies the Mazur-Ulam property (cf. \cite[Corollary 2]{Ding07}). Additional examples of Banach spaces satisfying the Mazur-Ulam property were provided by R. Liu and X.N. Fang and J.H. Wang. The list includes $c (\Gamma,\mathbb{R})$, $c_{0} (\Gamma,\mathbb{R})$, $\ell_{\infty} (\Gamma,\mathbb{R}),$ and the space $C(K,\mathbb{R})$ of all real-valued continuous functions on an arbitrary compact metric space $K$ (see \cite[Main Theorem and Corollary 6]{Liu2007} and \cite[Theorem 3.2]{FangWang06} for the result concerning $C(K,\mathbb{R})$). D. Tan showed that the spaces $L^{p}((\Omega, \Sigma, \mu), \mathbb{R}),$ of real-valued measurable functions on an arbitrary $\sigma$-finite measure space $(\Omega, \Sigma, \mu),$ satisfy the Mazur-Ulam property for all $1\leq p\leq \infty$ \cite{Ta:1,Ta:8,Ta:p}. Other references dealing with the Mazur-Ulam property can be found in \cite{Li2016,THL} and \cite{TL}.\smallskip

All previous examples of Banach spaces satisfying the Mazur-Ulam property are real sequence spaces and spaces of real-valued continuous or measurable functions.
However, it is an open and intriguing problem whether the spaces $\ell_{\infty}(\Gamma)$, $c_0(\Gamma)$ and $c(\Gamma)$ of complex sequences satisfy the Mazur-Ulam property or not. The same question is also open for spaces of complex-valued continuous functions on a compact metric space and for complex-valued measurable functions. Practically nothing is known in the complex setting. In this note we establish the first result in this direction by proving that the space $c_0(\Gamma)$ satisfies the Mazur-Ulam property, that is, for each Banach space $X$ every surjective isometry $\Delta: S(c_0(\Gamma))\to S(X)$ admits a unique extension to a surjective real linear isometry from $c_0(\Gamma)$ onto $X$ (see Theorem \ref{t c0 satisfies the Mazur-Ulam property}).\smallskip

The technical results and arguments developed to prove the Mazur-Ulam property for $c_0(\Gamma)$ are also valid for $\ell_\infty^m,$ and consequently $\ell_\infty^m$ satisfies the Mazur-Ulam property too.

\section{Supports and maximal convex subsets of the unit sphere}

Throughout this note, $X$ will be a complex Banach space, $\mathcal{B}_X$ will denote the closed unit ball of $X$, and $\Gamma$ will be an infinite set (equipped with the discrete topology). Following the standard notation, $c_0(\Gamma)$ will denote the Banach space of all functions $x:\Gamma \to \mathbb{C}$ such that, for all $\varepsilon >0$, the set $\{ n\in \Gamma : |x(n)|\geq \varepsilon \}$ is finite, while $\ell_{\infty}(\Gamma)$ will stand for the space of bounded functions from $\Gamma$ to $\mathbb{C}$. In the finite dimensional case, we shall write $\ell_{\infty}^m$ for $(\mathbb{C}^m,\|\cdot\|_\infty)$ with $m\in\mathbb{N}$. The symbol $\mathcal{L}^{\infty} (\Gamma)$ will stand for any of the spaces of complex sequences $\ell_{\infty}(\Gamma)$ or $c_0(\Gamma)$ equipped with the supremum norm. \smallskip

Henceforth, given an element $x\in \mathcal{L}^{\infty} (\Gamma),$ the symbol $x(k)$ will denote the $k$th component of $x$. For each $n\in \Gamma$ and each $\lambda\in \mathbb{T},$ we set $$A(n,\lambda):=\{ x\in S(\mathcal{L}^{\infty} (\Gamma)) : x(n) = \lambda\}$$ and $$\hbox{Pick}(n,\lambda) :=\{ x\in A(n,\lambda) : |x(k)|<1, \ \forall k\neq n\}.$$ Then $A(n,\lambda)$ is a maximal weak$^*$-closed proper face of $\mathcal{B}_{\mathcal{L}^{\infty} (\Gamma)}$ and a maximal convex subset of $S(\mathcal{L}^{\infty} (\Gamma))$.\smallskip

Our first result gathers a key property for later purposes.

\begin{lemma}\label{l existence of support functionals for the image of a face ellinfty} Let $\Delta : S(\mathcal{L}^{\infty} (\Gamma))\to S(X)$ be a surjective isometry. Then, for each $n\in \Gamma$ and each $\lambda\in \mathbb{T},$ the set $$\hbox{\rm supp}(n,\lambda) := \{\varphi\in X^* : \|\varphi\|=1,\hbox{ and } \varphi^{-1} (\{1\}) = \Delta(A(n,\lambda)) \}$$ is a non-empty weak$^*$-closed face of $\mathcal{B}_{\mathcal{L}^{\infty} (\Gamma)}$.
\end{lemma}

\begin{proof} The set $A(n,\lambda)$ is a maximal convex subset of $S(\mathcal{L}^{\infty} (\Gamma))$. Therefore, it follows from \cite[Lemma 5.1$(ii)$]{ChenDong2011} or \cite[Lemma 3.5]{Tan2014} that $\Delta(A(n,\lambda)) $ is a maximal convex subset of $X$. Thus, by Eidelheit's separation Theorem \cite[Theorem 2.2.26]{Megg98} there is a norm-one functional $\varphi\in X^*$ such that $\varphi^{-1} (\{1\}) = \Delta(A(n,\lambda))$ (compare \cite[Lemma 3.3]{Tan2016}). The rest can be straightforwardly checked by the reader.\smallskip
\end{proof}

We shall isolate next a property which was essentially shown in \cite[Lemmas 3.1 and 3.5]{FangWang06}, we include here an argument for completeness reasons. Following standard notation (compare \cite{ChenDong2011,FangWang06}), given a norm-one element $x$ in a Banach space $X$, we shall denote by St$(x)$ the \emph{star-like subset of $S(X)$ around $x$}, that is, the set given by $$\hbox{St} (x) :=\{y\in S(X) : \|x+ y \|=2\}.$$ It is known that St$(x)$ is precisely the union of all maximal convex subsets of $S(X)$ containing $x,$ and coincides with the set of all $y\in X$ such that the set $[x,y]:=\{t x + (1-t) y : t\in [0,1]\}$ is contained in $S(X)$.

\begin{lemma}\label{l face -1 face ellinfty} Let $\Delta : S(\mathcal{L}^{\infty} (\Gamma))\to S(X)$ be a surjective isometry. Then for each $n$ in $\Gamma$ and each $\lambda$ in $\mathbb{T}$ we have $\varphi \Delta(x) = -1$ for every $x$ in $A(n,-\lambda)$ and every $\varphi$ in $\hbox{\rm supp}(n,\lambda)$.
\end{lemma}

\begin{proof} Let us take $x\in A(n,-\lambda)$ and  $\varphi \in \hbox{supp}(n,\lambda)$. We can always pick $y$ in $\hbox{Pick}(n,\lambda)\cap c_0(\Gamma)$ (take, for example $y= \lambda e_n$). Clearly $-x\in  A(n,\lambda)$, and $$\|\Delta(x)-\Delta(y) \| = \|x-y \|=2,$$ and hence $- \Delta(x)\in \hbox{St} (\Delta(y)).$\smallskip

Now we argue as in \cite[Lemma 3.1]{FangWang06} to deduce that $\hbox{St} (\Delta(y)) = \Delta (A(n,\lambda))$. Namely, $z\in \hbox{St} (\Delta(y))$ if and only if $\| z + \Delta(y)\| = 2$, which by \cite[Corollary 2.2]{FangWang06} is equivalent to $\|\Delta^{-1} (z) + y\|=2 \Leftrightarrow \Delta^{-1} (z) \in \hbox{St} (y)= A(n,\lambda).$\smallskip

Therefore $- \Delta(x)\in \hbox{St} (\Delta(y))= \Delta (A(n,\lambda)),$ and thus, by definition, we have $$- \varphi ( \Delta(x))= \varphi (- \Delta(x))= 1.$$
\end{proof}

Additional properties of the sets supp$(n,\lambda)$ are established in the next lemma.

\begin{lemma}\label{l eproperties of supports ellinfty} Let $\Delta : S(\mathcal{L}^{\infty} (\Gamma))\to S(X)$ be a surjective isometry. Then the following statements hold:\begin{enumerate}[$(a)$]
\item For every $n_0, n_1$ in $\Gamma$ with $n_0\neq n_1$ and $\lambda,\mu\in \mathbb{T},$ we have $\hbox{\rm supp}(n_0,\lambda) \cap \hbox{\rm supp}(n_1,\mu)=\emptyset;$
\item Given $\mu,\nu$ in $\mathbb{T}$ and $n_0$ in $\Gamma$, we have $\hbox{\rm supp}(n_0,\nu) \cap \hbox{\rm supp}(n_0,\mu)\neq \emptyset$ if and only if $\mu=\nu$.
\end{enumerate}
\end{lemma}

\begin{proof}$(a)$ Arguing by contradiction we assume the existence of $\varphi\in \hbox{supp}(n_0,\lambda) \cap \hbox{supp}(n_1,\mu)$. Let us take two elements $y_0,y_1\in \mathcal{L}^{\infty} (\Gamma)$ such that $0\leq y_0,y_1\leq 1,$ $y_0 y_1 =0$ and $y_j (n_j)= 1$ for $j=0,1$. That is, $\lambda y_0\in A(n_0, \lambda)$ and $\mu y_1\in A(n_1, \mu)$.\smallskip

Since $- \mu y_1\in A(n_1, \mu)$, Lemma \ref{l face -1 face ellinfty} implies that  $\varphi \Delta (-\mu y_1) = -1.$ By definition $\varphi \Delta (\lambda y_0)= 1$, and then $$ 2 = \varphi \Delta (\lambda y_0) - \varphi \Delta (- \mu y_1) = |\varphi \Delta (\lambda y_0) - \varphi \Delta (- \mu y_1)  |$$ $$ \leq \| \Delta (\lambda y_0) - \Delta (- \mu y_1) \| =\|\lambda y_0+ \mu y_1\|=1,$$ which is impossible.\smallskip

$(b)$ As in $(a)$, let us take $\varphi\in \hbox{supp}(n_0,\nu) \cap \hbox{supp}(n_0,\mu),$ with $\mu \neq \nu$, and $y_0\in A(n_0,1)$. Since $\mu y_0\in A(n_0,\mu)$ and $\nu y_0\in A(n_0,\nu)$ we get $$2 = \varphi \Delta (\nu y_0) + \varphi \Delta ( \mu y_0) \leq \|\Delta (\nu y_0) + \Delta (\mu y_0) \|\leq 2.$$ By \cite[Corollary 2.2]{FangWang06} we have $2= \| \nu y_0  +  \mu y_0 \| = |\mu +\nu|,$ which holds if and only if $\mu =\nu$.
\end{proof}

Henceforth $e_n$ will denote the $n$th vector of the canonical basis of $\ell_{\infty}(\Gamma)$.

\begin{proposition}\label{p kernel of the support ellinfty} Let $\Delta : S(\mathcal{L}^{\infty} (\Gamma))\to S(X)$ be a surjective isometry. Let $n_0$ be an element in $\Gamma$ and let $\varphi$ be an element in $\hbox{\rm supp}(n_0,\lambda)$ where $\lambda\in \mathbb{T}$. Then $\varphi\Delta (x) =0$ for every $x\in S(\mathcal{L}^{\infty} (\Gamma))$ with $x(n_0)=0$. Furthermore, $|\varphi\Delta (x) |<1$ for every $x\in S(\mathcal{L}^{\infty} (\Gamma))$ with $|x(n_0)|<1.$
\end{proposition}

\begin{proof} Let us take $x\in S(\mathcal{L}^{\infty} (\Gamma))$ such that $x(n_0)=0$. The functions $x\pm \lambda e_{n_0}$ lie in $S(\mathcal{L}^{\infty} (\Gamma))$ with  $\lambda e_{n_0}\in A(n_0,\lambda)$ and $- \lambda e_{n_0}\in A(n_0,-\lambda)$. Let us fix $\varphi\in \hbox{supp}(n_0,\lambda)$. Lemma \ref{l face -1 face ellinfty} implies that $\varphi \Delta(-\lambda e_{n_0}) =-1$, and clearly $\varphi \Delta(\lambda e_{n_0}) =1$. Thus $$|\varphi \Delta(x) \pm 1|  = |\varphi \Delta(x) \pm \varphi \Delta(\lambda e_{n_0})| =  |\varphi \Delta(x) - \varphi \Delta(\mp \lambda e_{n_0})|$$ $$\leq \|\varphi\| \ \| \Delta(x) - \Delta(\mp \lambda e_{n_0}) \| = \| x  \pm \lambda e_{n_0} \| =1,$$ which assures that $\varphi \Delta(x)=0.$\smallskip

For the last statement, let us take $x\in S(\mathcal{L}^{\infty} (\Gamma))$ with $|x(n_0)|<1.$  Let us find $1>\varepsilon >0$ such that $|x(n_0)|< 1-\varepsilon$. We consider the non-empty set $C_{\varepsilon} :=\{ n\in \Gamma : |x(n)|\geq 1-\varepsilon \}$. Let $h\in S(\ell_{\infty} (\Gamma))$ be the characteristic function of the set ${C_{\varepsilon}}$. It is easy to check that $x h \in S(\mathcal{L}^{\infty} (\Gamma))$, $(x h) (n_0)=0$, and $\|x-x h\| < 1-\varepsilon <1$.\smallskip

Since $(x h) (n_0)=0,$ the first statement in this proposition proves that $\varphi \Delta (x h) =0$, and thus $$| \varphi \Delta (x)  | = | \varphi \Delta (x) -\varphi \Delta (x h) | \leq \| \Delta (x) - \Delta (x h )\| = \| x - x h\| < 1-\varepsilon <1.$$
\end{proof}

Let us discuss a consequence of the previous proposition. We fix $n_0\in \Gamma$ and $\lambda\in \mathbb{T}.$ For each $y\in \hbox{Pick}(n_0,\lambda)$ we know that $|\varphi \Delta (y)|< 1$ for all $\varphi \in  \hbox{supp}(n_1,\mu)$, with $n_1\in \Gamma\backslash\{n_0\}$ and $\mu\in \mathbb{T}$.

\section{The Mazur-Ulam property for $c_0 (\Gamma)$}

The next lemma is a particular case of \cite[Lemma 1]{Ding07} and \cite[Lemma 2.4]{Ta:8}. A proof is included here for completeness reasons.

\begin{proposition}\label{p antipodes ellinfty} Let $\Delta : S(\mathcal{L}^{\infty} (\Gamma))\to S(X)$ be a surjective isometry.
Then, for each $n\in \Gamma$ and each $\mu\in \mathbb{T}$ we have $\Delta (- \mu e_n) = -\Delta (\mu e_n)$.
\end{proposition}

\begin{proof} Let us find $x\in S(\mathcal{L}^{\infty} (\Gamma))$ satisfying $\Delta (x) = -\Delta (\mu e_n)$. By hypothesis
$$2=\|-2 \Delta (\mu e_n)\|=\|\Delta(x) -\Delta(\mu e_n) \|= \|x- \mu e_n\|,$$ which shows that $x(n) = -\mu$.\smallskip

Now, fix $m\neq n$ and take another $y\in S(\mathcal{L}^{\infty} (\Gamma))$ satisfying $\Delta (y) = -\Delta (\mu e_m)$. The above arguments also show that $y(m) = -\mu$. On the other hand, $$\|x-y\|=\|\Delta(x) -\Delta(y) \|=\|-\Delta (\mu e_n)+\Delta (\mu e_m)\|=\|\mu e_m- \mu e_n\|=1,$$ and hence $|x(m)+\mu|\leq 1$ and $|y(n)+\mu|\leq 1$.

Finally pick $z\in S(\mathcal{L}^{\infty} (\Gamma))$ satisfying $\Delta (z) = -\Delta (-\mu e_m)$. Under these assumption we know that $$\|z +\mu e_m \| = \|z-(-\mu e_m)\| = \|\Delta (z) - \Delta (-\mu e_m)\|= 2 \|- \Delta (-\mu e_m)\|=2,$$ witnessing that $z(m)=\mu$.\smallskip

Since $$\|x-z\|=\|\Delta(x) -\Delta(z) \|=\|-\Delta (\mu e_n)+\Delta (-\mu e_m)\|=\|-\mu e_m- \mu e_n\|=1,$$ and thus $|x(m)-\mu|\leq 1$ and $|-z(n)-\mu|\leq 1$.\smallskip

The inequalities $|x(m)+\mu|\leq 1$ and $|x(m)-\mu|\leq 1$ imply $x(m)=0$. Therefore, $x(m) =0$ for every $m\neq n$ and consequently $x= -\mu e_n$, which concludes the proof.
\end{proof}

\begin{remark}\label{remark antipodes finite dimensional}{\rm If in the previous proposition $\mathcal{L}^{\infty} (\Gamma)$ is replaced with $\ell_\infty^m$ then, the same conclusion remains true by the original Tingley's theorem \cite{Ting1987}, which shows that for finite dimensional normed spaces $X$ and $Y$, every surjective isometry $\Delta: S(X)\to S(Y)$ satisfies $\Delta(-x) = -\Delta(x)$ for every $x\in S(X)$.}
\end{remark}

The next proposition establishes the behavior of a surjective isometry on a spherical multiple of some element of the canonical basis.

\begin{proposition}\label{p imaginary ellinfty} Let $\Delta : S(\mathcal{L}^{\infty} (\Gamma))\to S(X)$ be a surjective isometry.
Then, for each $n\in \Gamma$ and each $\lambda\in \mathbb{T}$ we have $\Delta (\lambda e_n) \in \{\lambda \Delta (e_n), \overline{\lambda} \Delta (e_n)\}$.\smallskip

Furthermore, if for some $n\in \Gamma$ we have $\Delta (\lambda e_n) = \lambda \Delta (e_n)$ {\rm(}respectively, $\Delta (\lambda e_n) = \overline{\lambda} \Delta (e_n)${\rm)} for some $\lambda \in \mathbb{T}\backslash \{\pm 1\}$, then $\Delta (\mu e_n) = \mu \Delta (e_n)$ {\rm(}respectively, $\Delta (\mu e_n) = \overline{\mu} \Delta (e_n)${\rm)} for all $\mu \in \mathbb{T}$.
\end{proposition}

\begin{proof} The element $\lambda \Delta(e_n)$ lies in $S(X)$, so by the surjectivity of $\Delta$ there exists $x\in S(\mathcal{L}^{\infty} (\Gamma))$ such that $\Delta(x) = \lambda \Delta(e_n)$. We shall first prove that $x(k)=0$ for all $k\neq n$.\smallskip

Suppose that $|x(k)|=1$ for some $k\neq n$. Then, by Lemma \ref{l existence of support functionals for the image of a face ellinfty}, we get $\varphi \Delta(x) = 1$ for all $\varphi\in \hbox{supp}(k,x(k)).$ However $1=\varphi\Delta(x) = \varphi (\lambda \Delta(e_n)) = \lambda \varphi \Delta (e_n),$ and since $e_n (k)=0$, Proposition \ref{p kernel of the support ellinfty} gives $\varphi \Delta (e_n) =0$, which is impossible. We have therefore shown that $|x(k)|<1$ for $k\neq n$.\smallskip

If $0<|x(k)|<1,$ let $y \in S(\mathcal{L}^{\infty} (\Gamma))$ be the element defined by $y(m) = x(m)$ for all $m\neq k$ and $y(k) = x(k)/{|x(k)|} $. Clearly $$\|x-y\|= \left| x(k) - \frac{1}{|x(k)|} x(k)\right| = \left|\frac{|x(k)|-1}{|x(k)|}\right| |x(k)| = 1- |x(k)|,$$ and for each $\varphi \in \hbox{supp}(k,{x(k)}/{|x(k)|}),$ Lemma \ref{l existence of support functionals for the image of a face ellinfty} assures that $\varphi \Delta (y)=1$. Therefore $$1- |\varphi \Delta (x)| = \varphi \Delta (y) - |\varphi \Delta (x)| \leq |\varphi \Delta (y) - \varphi \Delta (x)|\leq \| y-x\| = 1- |x(k)|,$$ which shows that   $$ |x(k)| \leq |\varphi \Delta (x)| = |\varphi (\lambda \Delta (e_n))| = |\varphi  \Delta (e_n)| = \hbox{(by Proposition \ref{p kernel of the support ellinfty})}= 0,$$ leading to a contradiction. This shows that $\lambda \Delta (e_n)=\Delta (\mu e_n)$ for some $\mu\in \mathbb{T}$.\smallskip

By applying Proposition \ref{p antipodes ellinfty} we get $$|\lambda \pm 1|=\| \lambda \Delta(e_n) \pm \Delta(e_n)\|=\| \Delta(x) \pm \Delta(e_n)\| $$ $$=\| \Delta(\mu e_n) \pm \Delta(e_n)\|= \| \mu e_n \pm e_n\| =|\mu \pm 1|,$$ and then $\mu \in \{\lambda, \overline{\lambda}\}.$\smallskip

Let us finally prove the last statement. Let us assume that $\Delta (\lambda e_n) = \lambda \Delta (e_n)$ {\rm(}respectively, $\Delta (\lambda e_n) = \overline{\lambda} \Delta (e_n)$\textrm{)} for some $\lambda \in \mathbb{T}\backslash \{\pm 1\}$. Let $\mu$ be an arbitrary element in $\mathbb{T}\backslash \mathbb{R}$. We have shown above that $\Delta (\mu e_n) = \mu \Delta (e_n)$ or $\Delta (\mu e_n) = \overline{\mu} \Delta (e_n)$. We shall prove that the second possibility (respectively, the first one) is impossible. Arguing by contradiction, we suppose that $\Delta (\mu e_n) = \overline{\mu} \Delta (e_n)$ (respectively, $\Delta (\mu e_n) = {\mu} \Delta (e_n)$). By the assumptions and Proposition \ref{p antipodes ellinfty} we have $$|\lambda + \mu | = \| \lambda \Delta (e_n) + \mu \Delta (e_n) \| = \|  \Delta (\lambda e_n) + \Delta (\overline{\mu} e_n) \|$$ $$= \|  \Delta (\lambda e_n) - \Delta (-\overline{\mu} e_n) \| = \|  \lambda e_n + \overline{\mu} e_n \| = |\lambda + \overline{\mu} |$$ (respectively, $$|\lambda + \mu | = \| \lambda \Delta (e_n) + \mu \Delta (e_n) \| = \|  \Delta (\overline{\lambda} e_n) + \Delta ({\mu} e_n) \| = \|  \overline{\lambda} e_n + {\mu} e_n \| = |\overline{\lambda} + {\mu} |).$$ Any of the previous identities holds if and only if $$2 + 2 \Re\hbox{e} (\lambda \overline{\mu})=|\lambda|^2 +| \mu |^2 + 2 \Re\hbox{e} (\lambda \overline{\mu}) = |\lambda + \mu |^2 = |\lambda + \overline{\mu} |^2 = 2 + 2 \Re\hbox{e} (\lambda {\mu}),$$ equivalently $$\Re\hbox{e} (\lambda) \Re\hbox{e} ({\mu}) + \Im\hbox{m} (\lambda) \Im\hbox{m} (\mu)= \Re\hbox{e} (\lambda) \Re\hbox{e} ({\mu}) - \Im\hbox{m} (\lambda) \Im\hbox{m} (\mu),$$ which is impossible because $\lambda,\mu\notin \mathbb{R}$.
\end{proof}

Let $\Delta : S(\mathcal{L}^{\infty} (\Gamma))\to S(X)$ be a surjective isometry. Henceforth, we set  $$\Gamma_1^{\Delta} :=\{n\in \Gamma : \Delta (\lambda e_n) = \lambda \Delta (e_n)\hbox{ for all }\lambda \in \mathbb{T}\},$$ and $$\Gamma_2^{\Delta} :=\{n\in \Gamma : \Delta (\lambda e_n) = \overline{\lambda} \Delta (e_n)\hbox{ for all }\lambda \in \mathbb{T}\}.$$ It follows from Proposition \ref{p imaginary ellinfty} that $\Gamma = \Gamma_1^{\Delta} \stackrel{\circ}{\cup} \Gamma_2^{\Delta}$. Given $n\in \Gamma_1^{\Delta}$ (respectively, $n\in \Gamma_2^{\Delta}$) and $\alpha\in \mathbb{C}$ we define $\sigma_{n} (\alpha) = \alpha$ (respectively, $\sigma_{n} (\alpha) = \overline{\alpha}$). We know from Proposition \ref{p imaginary ellinfty} that \begin{equation}\label{eq Delta takes out complex or conjugate scalar at a coordinate} \Delta (\lambda e_n) = \sigma_{n} (\lambda) \ \Delta(e_n),\ \forall \lambda\in \mathbb{T}, n\in \Gamma.
\end{equation} We also observe that $\sigma_{n} (\alpha) = \alpha$ for all $\alpha\in \mathbb{R}$, $n\in \Gamma$.\smallskip

Let $x$ and $y$ be two vectors in a (real or complex) normed space $X$. The elements $x,y$ are said to be \emph{$M$-orthogonal} (denoted by
$x\perp_{M} y$) if $\| x \pm y\| = \max \{\|x\| , \|y\|\}$. In the setting of complex Banach spaces we can find more variants of geometric orthogonality. Accordingly to the notation in \cite{OikPeRa11}, we shall say
that $x$ and $y$ are {\it completely $M$-orthogonal} (denoted by $x \perp_{_{CM}} y$) if $$\|\alpha x +\beta y\| = \max\{|\alpha| \ \|x\|, |\beta|\ \|y\|\},$$ for every $\alpha,\beta$ in $\mathbb{C}$.\smallskip

The canonical notion of (algebraic) orthogonality in $\ell_{\infty}(\Gamma),$ $c_{0}(\Gamma)$ and $c(\Gamma)$ reads as follows: elements $a$, $b$ in any of these spaces are said to be \emph{orthogonal} or \emph{disjoint} if $a b=0$. Algebraic orthogonality is stronger than complete $M$-orthogonality and the latter is stronger than $M$-orthogonality. For example, $x=(1,1/2,0)$ and $y=(0,1/4,1)$ are $M$-orthogonal and completely $M$-orthogonal with $x y \neq 0$ in $\ell_{\infty}^3$, the three dimensional $\ell_{\infty}$-space. While $x=(1,1/2,0)$ and $y=(0,{\sqrt{3} i }/{2},1)$ are $M$-orthogonal but not completely $M$-orthogonal.\smallskip

Finally, we shall say that a set $\{x_1,\ldots,x_m\}$ in $X$ is {\it completely $M$-orthogonal} if $$\Big\|\sum_{j=1}^m \alpha_j x_j \Big\| = \max\{|\alpha_j| \ \|x_j\| : 1\leq j\leq m\},$$ for every $\alpha_1,\ldots,\alpha_m$ in $\mathbb{C}$.\smallskip

The following technical lemma will be required latter.

\begin{lemma}\label{l characterization CMorthog in sphere} Let $\{x_1,\ldots,x_m\}$ be a subset of the unit sphere of a complex normed space $X$. Then $\{x_1,\ldots,x_m\}$ is completely $M$-orthogonal if and only if the equality $$\displaystyle \Big\|\sum_{j=1}^m \alpha_j x_j \Big\| = 1$$ holds for every $\alpha_1,\ldots,\alpha_m$ in $\mathbb{T}$.
\end{lemma}

\begin{proof} The ``only if'' implication is clear. For the ``if'' implication, let us regard $X$ as a closed subspace of $C(\mathcal{B}_{X^*})$, the space of all complex-valued continuous functions on the closed unit ball of $X^*$. Under this identification, $x_1,\ldots,x_m$ are norm-one functions in $C(\mathcal{B}_{X^*})$ such that $\displaystyle \Big\|\sum_{j=1}^m \alpha_j x_j \Big\| = 1$ for every $\alpha_1,\ldots,\alpha_m$ in $\mathbb{T}$. It is easy to see that this is equivalent to say that $\displaystyle \sum_{j=1}^m |x_j(t)| \leq 1$ for every $t\in \mathcal{B}_{X^*}$. Therefore given $t_0$ in $\mathcal{B}_{X^*}$ and $j_0\in \{1,\ldots, m\}$ with $|x_{j_0} (t_0)|=1$ then $x_k(t_0)=0$ for all $k\neq j_0$.\smallskip

Take $\beta_1,\ldots,\beta_m$ in $\mathbb{C}\backslash \{0\}$, with $|\beta_{j_0} |=\max\{ |\beta_j| : 1\leq j \leq m\}$ for an index $j_0$. Since $$\Big\| \sum_{j=1}^m \beta_j x_j  \Big\|= |\beta_{j_0}| \ \Big\| x_{j_0} + \sum_{j\neq j_0}^m \frac{\beta_j}{\beta_{j_0}} x_j  \Big\|$$ $$ \leq |\beta_{j_0}| \ \left( \|x_{j_0}\| + \sum_{j\neq j_0}^m \left|\frac{\beta_j}{\beta_{j_0}}\right| \|x_j\|\right) \leq |\beta_{j_0}| \  \sum_{j=1}^m \|x_j\| \leq |\beta_{j_0}|,$$ we deduce that $\displaystyle \Big\| \sum_{j=1}^m \beta_j x_j  \Big\| \leq \max\{ |\beta_j| : 1\leq j \leq m\}$. Now, by taking $t_0\in \mathcal{B}_{X^*}$ with $|x_{j_0}(t_0)|=1$, we get $\displaystyle \Big\| \sum_{j=1}^m \beta_j x_j  \Big\| \geq \Big|\sum_{j=1}^m \beta_j x_j  (t_0) \Big|= |\beta_{j_0} | \ |x_{j_0} (t_0)| =|\beta_{j_0}|.$\end{proof}

\begin{corollary}\label{c orthogonal have CM orthog images} Let $\Delta : S(\mathcal{L}^{\infty}(\Gamma))\to S(X)$ be a surjective isometry.
Then, for each $n,m\in \Gamma$ with $n\neq m$ we have $\Delta ( e_n) \perp_{_{CM}} \Delta (e_m)$.
\end{corollary}

\begin{proof} Take $\lambda,\mu$ in $\mathbb{T}$. By applying \eqref{eq Delta takes out complex or conjugate scalar at a coordinate} and Proposition \ref{p imaginary ellinfty} we get $$ \| \lambda \Delta ( e_n) +\mu \Delta (e_m) \| = \|  \Delta (\sigma_n(\lambda) e_n) + \Delta (\sigma_m(\mu) e_m) \| $$ $$= \| \Delta (\sigma_n(\lambda) e_n) - \Delta (-\sigma_m(\mu) e_m) \| = \| \sigma_n(\lambda) e_n  +\sigma_m(\mu) e_m \| = 1.$$ The desired conclusion follows from Lemma \ref{l characterization CMorthog in sphere}.
\end{proof}

We shall establish next a series of strengthened versions and consequences of the above corollary.

\begin{proposition}\label{p orthogonal have CM orthog images 2} Let $\Delta : S(\mathcal{L}^{\infty}(\Gamma))\to S(X)$ be a surjective isometry.
Let $n_1,\ldots,n_k$ be different elements in $\Gamma$. Then the identity $$\sum_{j=1}^{k}  \Delta(\alpha_j e_{n_j} ) = \Delta \left(\sum_{j=1}^k \alpha_j e_{n_j} \right)$$ holds for every $\alpha_{1},\ldots, \alpha_k$ in $\mathbb{T}.$ As a consequence, the set $\{\Delta(e_{n_1}),\ldots,\Delta( e_{n_{k}})\}$ is completely $M$-orthogonal, that is, $$\left\|\sum_{j=1}^k \alpha_j \Delta ( e_{n_j}) \right\| = \max \{|\alpha_j| : j=1,\ldots,k \},$$ for every $\alpha_1,\ldots,\alpha_k$ in $\mathbb{C}.$
\end{proposition}

\begin{proof} We shall argue by induction on $k$. The case $k=1$ is clear. Let us assume that the desired statement is true for $k-1$ with $k\geq 2$. We shall first show that \begin{equation}\label{eq 1 22 08} \left\|\sum_{j=1}^{k}  \Delta(\alpha_j e_{n_j} ) \right\| = 1.
\end{equation}

By the induction hypothesis we have $$ \sum_{j=1}^{k}  \Delta(\alpha_j e_{n_j} ) = \Delta \left(\sum_{j=1}^{k-1} \alpha_j e_{n_j}\right) + \Delta(\alpha_k e_{n_k}),$$ and thus, by the assumptions on $\Delta$ and Proposition \ref{p imaginary ellinfty} (see also \eqref{eq Delta takes out complex or conjugate scalar at a coordinate}) we deduce that  $$\left\|\sum_{j=1}^{k}  \Delta(\alpha_j e_{n_j} ) \right\|\! = \!\left\| \Delta \left(\sum_{j=1}^{k-1} \alpha_j e_{n_j}\right) + \Delta(\alpha_k e_{n_k})\right\|\!=\! \left\| \Delta \left(\sum_{j=1}^{k-1} \alpha_j e_{n_j}\right) - \Delta(-\alpha_k e_{n_k})\right\|$$ $$=\left\| \sum_{j=1}^{k-1} \alpha_j e_{n_j} + \alpha_k e_{n_k}\right\| = 1,$$ which proves the claim in \eqref{eq 1 22 08}.\smallskip

Since, by \eqref{eq 1 22 08}, $\displaystyle \sum_{j=1}^{k}  \Delta(\alpha_j e_{n_j} )\in S(X),$ there exists $x\in S(\mathcal{L}^{\infty}(\Gamma))$ satisfying $$\Delta(x) =  \displaystyle \sum_{j=1}^{k}  \Delta(\alpha_j e_{n_j} ) = \Delta \left(\sum_{j=1}^{k-1} \alpha_j e_{n_j}\right) + \Delta(\alpha_k e_{n_k}).$$ To prove the first statement, it suffices to show that \begin{equation}\label{eq 2 22 08} x= \sum_{j=1}^{k} \alpha_j e_{n_j}.
\end{equation} To this end, pick an arbitrary $m\in \Gamma\backslash \{n_1,\ldots, n_k\}$. If $|x(m)|=1$, we take $\phi\in \hbox{supp}(m, x(m))$. By construction $\phi \Delta (x) = 1$. However, having in mind that $\alpha_j e_{n_j}(m)=0$ for all $j$, an application of Proposition \ref{p kernel of the support ellinfty} implies that $$1=\phi \Delta (x) = \phi \left(\sum_{j=1}^{k}  \Delta(\alpha_j e_{n_j} )\right) = \sum_{j=1}^{k}  \phi\Delta(\alpha_j e_{n_j} )=0,$$ which is impossible.\smallskip

Suppose now that $0<|x(m)|<1.$ Let $y \in S(\mathcal{L}^{\infty} (\Gamma))$ be the element defined by $y(n) = x(n)$ for all $n\neq m,$ and $y(m) = x(m)/{|x(m)|}$. Clearly $\|x-y\| = 1- |x(m)|,$ and given $\phi \in \hbox{supp}(m, x(m)/{|x(m)|}),$ Lemma \ref{l existence of support functionals for the image of a face ellinfty} implies that $\phi \Delta (y)=1$. Combining all these facts together we have $$1- |\phi \Delta (x)| = \phi \Delta (y) - |\phi \Delta (x)| \leq |\phi \Delta (y) - \phi \Delta (x)|\leq \| y-x\| = 1- |x(m)|,$$ which shows that   $$0< |x(m)| \leq |\phi \Delta (x)| = \left|\phi \left( \sum_{j=1}^{k}  \Delta(\alpha_j e_{n_j} ) \right)\right|  $$ $$\leq  \sum_{j=1}^{k}  \left|\phi \left( \Delta (\alpha_j e_{n_j})\right)\right|= \hbox{(again by Proposition \ref{p kernel of the support ellinfty})}= 0,$$ providing a contradiction. Therefore $x(m)=0,$ for every $m\in \Gamma\backslash \{n_1,\ldots, n_k\}$. We have therefore shown that $\displaystyle x= \sum_{j=1}^{k} \mu_j e_{n_j},$ with $\mu_1,\ldots,\mu_k\in \mathbb{C}$ and $\max\{|\mu_j|:1\leq j\leq k\}=1$.\smallskip

Since $|\alpha_{l}|=1$ for every $1\leq l\leq k$, given $\phi \in \hbox{supp}(n_l,\alpha_l),$ we deduce from Lemma \ref{l existence of support functionals for the image of a face ellinfty} and Proposition \ref{p kernel of the support ellinfty} that $$2\geq \left\| \Delta(x) +  \Delta (\alpha_{l} e_{n_{l}}) \right\| \geq\left|\phi (\Delta(x) + \Delta (\alpha_{l} e_{n_{l}})) \right|$$ $$= \left|\sum_{j=1}^k \phi \left( \Delta (\alpha_j e_{n_j})\right) +\phi (\Delta (\alpha_{l} e_{n_{l}}))\right| = 2 \left|\phi (\Delta (\alpha_{l} e_{n_{l}}))\right|=2,$$ and it follows from the assumptions and Proposition \ref{p antipodes ellinfty} that
$$ 2= \left\| \Delta(x) +  \Delta (\alpha_{l} e_{n_{l}}) \right\| =  \left\| x +\alpha_{l}\ e_{n_{l}} \right\| = \max\{ |\mu_j| : j\neq l\}\vee |\mu_{l}+\alpha_{l}|,$$ which implies that $$ |\mu_{l}+ \alpha_{l}| =2,$$ and thus $\alpha_{l}= \mu_{l}$ for every $1\leq l\leq k$, which concludes the proof of \eqref{eq 2 22 08}.\smallskip

To prove the last affirmation, let $\alpha_1,\ldots,\alpha_k$ be arbitrary elements in $\mathbb{T}$. Applying Proposition \ref{p imaginary ellinfty} (see \eqref{eq Delta takes out complex or conjugate scalar at a coordinate}) and \eqref{eq 1 22 08} we have $$\left\|\sum_{j=1}^k \alpha_k \Delta(e_{n_k}) \right\|=  \left\|\sum_{j=1}^k  \Delta(\sigma_k (\alpha_k) e_{n_k}) \right\| = 1.$$ Finally, Lemma \ref{l characterization CMorthog in sphere} gives the desired conclusion.
\end{proof}

One more technical result is separating us from our first main goal.

\begin{proposition}\label{p algebraic ellinfty} Let $\Delta : S(\mathcal{L}^{\infty} (\Gamma))\to S(X)$ be a surjective isometry.
Then, for every $n_1,\ldots, n_k\in \Gamma$ and every $\lambda_1,\ldots,\lambda_k\in \mathbb{C}\backslash\{0\}$ with $\max\{|\lambda_1|,\ldots,|\lambda_k|\}=1$,
we have $$\Delta \left(\sum_{j=1}^{k} \lambda_j e_{n_j}\right) = \sum_{j=1}^{k} \sigma_{n_j}(\lambda_j) \Delta (e_{n_j}).$$
\end{proposition}

\begin{proof} Proposition \ref{p orthogonal have CM orthog images 2} guarantees that the set $\{\Delta (e_{n_1}),\ldots, \Delta (e_{n_k})\}\subset S(X)$ is completely $M$-orthogonal, therefore $\displaystyle \sum_{j=1}^{k} \sigma_{n_j}(\lambda_j) \Delta (e_{n_j})$ is an element in the unit sphere of $X$. By the surjectivity of $\Delta$ there exists $x$ in  $S(\mathcal{L}^{\infty} (\Gamma))$ satisfying $\Delta (x) = \displaystyle \sum_{j=1}^{k} \sigma_{n_j}(\lambda_j) \Delta (e_{n_j})$. We shall show that $\displaystyle x=\sum_{j=1}^{k} \lambda_j e_{n_j}$.\smallskip

In the next four paragraphs we shall follow arguments similar to those in the proof of Proposition \ref{p imaginary ellinfty} (see also \eqref{eq Delta takes out complex or conjugate scalar at a coordinate}). Take $m\in \Gamma\backslash \{n_1,\ldots, n_k\}$. If $|x(m)|=1$, we take $\phi\in \hbox{supp}(m, x(m))$. Lemma \ref{l existence of support functionals for the image of a face ellinfty} implies that $\phi \Delta (x) = 1$. However, having in mind that $e_{n_j} (m) =0$ for every $j$, Proposition \ref{p kernel of the support ellinfty} applies to prove that $$1=\phi \Delta (x) = \phi \left(\sum_{j=1}^{k} \sigma_{n_j}(\lambda_j)\ \Delta (e_{n_j})\right) = \sum_{j=1}^{k} \sigma_{n_j}(\lambda_j)\ \phi (\Delta (e_{n_j}))=0,$$ which is impossible.\smallskip

Suppose now that $0<|x(m)|<1.$ Let $y \in S(\mathcal{L}^{\infty} (\Gamma))$ be the element defined by $y(n) = x(n)$ for all $n\neq m$ and $y(m) = x(m)/{|x(m)|}$. Clearly $\|x-y\| = 1- |x(m)|.$ Given $\phi \in \hbox{supp}(m, x(m)/{|x(m)|}),$ Lemma \ref{l existence of support functionals for the image of a face ellinfty} assures that $\phi \Delta (y)=1$. Combining all these facts together we have $$1- |\phi \Delta (x)| = \phi \Delta (y) - |\phi \Delta (x)| \leq |\phi \Delta (y) - \phi \Delta (x)|\leq \| y-x\| = 1- |x(m)|,$$ which shows that   $$0< |x(m)| \leq |\phi \Delta (x)| = \left|\phi \left( \sum_{j=1}^{k} \sigma_{n_j}(\lambda_j) \Delta (e_{n_j})\right)\right|  $$ $$\leq  \sum_{j=1}^{k} |\sigma_{n_j}(\lambda_j)| \ \left|\phi \left( \Delta (e_{n_j})\right)\right|=  \hbox{(again by Proposition \ref{p kernel of the support ellinfty})}= 0,$$ providing a contradiction. Therefore $x(m)=0$ for every $m\in \Gamma\backslash \{n_1,\ldots, n_k\}$.\smallskip

We have shown that $\displaystyle x=\sum_{j=1}^{k} \mu_j e_{n_j}$ and $\displaystyle \Delta \left(\sum_{j=1}^{k} \mu_j e_{n_j}\right) = \sum_{j=1}^{k} \sigma_{n_j}(\lambda_j) \Delta (e_{n_j}),$ where $\mu_1,\ldots, \mu_k\in \mathbb{C}$ with $\max\{|\mu_j|: j=1,\ldots,k\}=1$.\smallskip

Our next goal is to prove that $\mu_j=\lambda_j$ for all $j\in\{1,\ldots,k\}$.\smallskip

We consider the spectrum of $x$, $\sigma (x) =\{ \mu_1,\ldots, \mu_k \}$. By a little abuse of notation the set  $\sigma (\Delta(x)) =\{ \sigma_{n_1}(\lambda_1),\ldots, \sigma_{n_k}(\lambda_k) \}$ will be called the spectrum of $\Delta(x)$. Up to an appropriate reordering we may assume that $$1=|\lambda_1|=\ldots =|\lambda_{m}| \geq |\lambda_{m+1}|\geq  \ldots \geq |\lambda_k|.$$

In a first step, let us take an index $j_0$ such that $|\lambda_{j_0}|=1$. It follows from the assumptions, Proposition \ref{p orthogonal have CM orthog images 2} and Proposition \ref{p imaginary ellinfty} (see also \eqref{eq Delta takes out complex or conjugate scalar at a coordinate}) that
$$ 2= \left\| \Delta(x) + \sigma_{n_{j_0}} (\lambda_{j_0}) \Delta (e_{n_{j_0}}) \right\|=\left\| \Delta(x) +  \Delta (\lambda_{j_0} e_{n_{j_0}}) \right\| =  \left\| x +\lambda_{j_0} e_{n_{j_0}} \right\| $$ $$= \max\{ |\mu_j| : j\neq j_0\}\vee |\mu_{j_0}+\lambda_{j_0}|,$$ which implies that $$ |\lambda_{j_0}+ \mu_{j_0}| =2,$$ and thus $\mu_{j_0}= \lambda_{j_0}$. We have shown that \begin{equation}\label{coeficients coincide with modulus one reverse} \mu_{j}=\lambda_{j}, \hbox{ for all $j\in\{1,\ldots,k\}$ with $|\lambda_{j}|=1$.}
\end{equation}

On the other hand, let us choose an index $j_0$ such that $|\mu_{j_0}|=1$. It follows from Proposition \ref{p orthogonal have CM orthog images 2} and Proposition \ref{p imaginary ellinfty} (see also \eqref{eq Delta takes out complex or conjugate scalar at a coordinate}) that $$\max\{ |\sigma_{n_j}(\lambda_j)| : j\neq j_0\}\vee |\sigma_{n_{j_0}}(\lambda_{j_0})+\sigma_{n_{j_0}} (\mu_{j_0})|=\left\| \Delta(x) + \sigma_{n_{j_0}} (\mu_{j_0}) \Delta (e_{n_{j_0}}) \right\|$$ $$ = \left\| \Delta(x) +  \Delta (\mu_{j_0}\ e_{n_{j_0}}) \right\|= \left\| \Delta(x) -  \Delta (-\mu_{j_0}\ e_{n_{j_0}}) \right\| =  \left\| x +\mu_{j_0}\ e_{n_{j_0}} \right\|= 2 |\mu_{j_0}| = 2,$$ which implies that $$ |\sigma_{n_{j_0}}(\lambda_{j_0})+\sigma_{n_{j_0}} (\mu_{j_0})| =2,$$ and thus $\sigma_{n_{j_0}}(\lambda_{j_0})= \sigma_{n_{j_0}} (\mu_{j_0})$, or equivalently $\lambda_{j_0}= \mu_{j_0}$. We have shown that \begin{equation}\label{coeficients coincide with modulus one} \lambda_{j}= \mu_{j}, \hbox{ for all $j\in\{1,\ldots,k\}$ with $|\mu_{j}|=1$.}
\end{equation}

Therefore  $\sigma (x) =\{ \lambda_1,\ldots,\lambda_{m},\mu_{m+1},\ldots \mu_k \}$ with $|\lambda_1|=\ldots=|\lambda_{m}|=1,$ and $|\mu_j| <1$ for all $j\geq m+1$.\smallskip

We claim that $\mu_j\neq 0$ for every $j\in\{1,\ldots,k\}$. Namely, we already know that $\mu_j\neq 0$ for all $1\leq j\leq m$. Arguing by contradiction, we assume the existence of $j_0\in \{m+1,\ldots, k\}$ such that $\mu_{j_0} =0$. Take the element $\displaystyle z=-\frac{\lambda_{j_0}}{|\lambda_{j_0}|} e_{n_{j_0}}+ \sum_{j=1}^m \lambda_j e_{n_j}$. Proposition \ref{p orthogonal have CM orthog images 2} assures that $$\Delta(z) = -\frac{\sigma_{n_{j_0}} (\lambda_{j_0})}{|\lambda_{j_0}|} \Delta(e_{n_{j_0}})+\sum_{j=1}^{m} \sigma_{n_j}(\lambda_j) \Delta (e_{n_j}).$$ Another application of Proposition \ref{p orthogonal have CM orthog images 2} gives $$1<1+|\lambda_{j_0}|=(1+|\lambda_{j_0}|)\vee \max\{|\lambda_{j}|: m+1\leq j \leq k, j\neq j_0\} $$ $$=\left\| \sum_{j=1}^{k} \sigma_{n_j}(\lambda_j) \Delta (e_{n_j}) -  \left(-\frac{\sigma_{n_{j_0}} (\lambda_{j_0})}{|\lambda_{j_0}|} \Delta(e_{n_{j_0}})+\sum_{j=1}^{m} \sigma_{n_j}(\lambda_j) \Delta (e_{n_j})\right) \right\|$$ $$=\|\Delta(x) -\Delta (z)\|=\|x-z\|= \left|\frac{\lambda_{j_0}}{|\lambda_{j_0}|}\right|\vee \max\{|\mu_{j}|: m+1\leq j \leq k, j\neq j_0\}=1,$$ which is impossible. Therefore $\mu_j\neq 0$ for every $j$.\smallskip

We shall next prove that $\lambda_{m+1} = \mu_{m+1}$.\smallskip

We pick now $j_0\in \{m+1,\ldots, k\}$. In this case $0<|\mu_{j_0}|<1$. We know that $$\max\{ |\sigma_{n_j}(\lambda_j)| : j\neq j_0\}\vee \left|\sigma_{n_{j_0}}(\lambda_{j_0})+\frac{\sigma_{n_{j_0}} (\mu_{j_0})}{|\mu_{j_0}|} \right|$$ $$=\left\| \Delta(x) + \frac{\sigma_{n_{j_0}} (\mu_{j_0})}{|\mu_{j_0}|} \Delta (e_{n_{j_0}}) \right\|=\left\| \Delta(x) +  \Delta \left(\frac{\mu_{j_0}}{|\mu_{j_0}|} e_{n_{j_0}}\right) \right\| $$ $$=\left\| x +  \frac{\mu_{j_0}}{|\mu_{j_0}|} e_{n_{j_0}} \right\|=\left| \mu_{j_0} +  \frac{\mu_{j_0}}{|\mu_{j_0}|} \right| = 1+ |\mu_{j_0}|>1,$$ which implies that
\begin{equation}\label{eq distance 2}  1+ |\mu_{j_0}| =\left|\sigma_{n_{j_0}}(\lambda_{j_0})+\frac{\sigma_{n_{j_0}} (\mu_{j_0})}{|\mu_{j_0}|} \right|= \left|\lambda_{j_0}+\frac{\mu_{j_0}}{|\mu_{j_0}|} \right|
\leq \left|\lambda_{j_0}\ \right|+ 1,
\end{equation} and hence $|\mu_{j_0}| \leq \left|\lambda_{j_0}\ \right|,$ for every $j_0 \geq m+1$.\smallskip

Now taking $y\in S(\mathcal{L}^{\infty} (\Gamma))$ such that $$\displaystyle \Delta (y) = - \sigma_{n_{m+1}}(\lambda_{m+1}) \Delta (e_{n_{m+1}})+ \sum_{ m+1\neq j =1}^k \sigma_{n_j}(\lambda_j) \Delta (e_{n_j}),$$ we know from the above arguments that $\displaystyle y=\sum_{j=1}^{k} \gamma_j e_{n_j}$ and $$\Delta \left(\sum_{j=1}^{k} \gamma_j e_{n_j}\right) = - \sigma_{n_{m+1}}(\lambda_{m+1}) \Delta (e_{n_{m+1}})+ \sum_{ m+1\neq j =1}^k \sigma_{n_j}(\lambda_j) \Delta (e_{n_j}),$$ where $\gamma_1,\ldots, \gamma_k\in \mathbb{C}\backslash \{0\}$ with $\max\{|\gamma_j|: j=1,\ldots,k\}=1$, $\gamma_j = \lambda_j$ for all $1\leq j\leq m,$ $0<|\gamma_j|<1$ for all $m+1\leq j\leq k$ (compare the arguments leading to \eqref{coeficients coincide with modulus one reverse} and \eqref{coeficients coincide with modulus one}), and by Proposition \ref{p orthogonal have CM orthog images 2} $$\max\{ |\sigma_{n_j}(\lambda_j)| : j\neq m+1\}\vee \left|-\sigma_{n_{m+1}}(\lambda_{m+1})+\frac{\sigma_{n_{m+1}} (\gamma_{m+1})}{|\gamma_{m+1}|} \right|$$ $$= \left\|\Delta(y) + \frac{\sigma_{n_{m+1}} (\gamma_{m+1})}{|\gamma_{m+1}|} \Delta (e_{n_{m+1}}) \right\| = \left\|y + \frac{\gamma_{m+1}}{|\gamma_{m+1}|} e_{n_{m+1}} \right\|$$ $$ = \left|\gamma_{m+1}+\frac{\gamma_{m+1}}{|\gamma_{m+1}|} \right| = 1+ |\gamma_{m+1}|>1,$$ and thus \begin{equation}\label{eq 2808 m+1} \left|-\lambda_{m+1}+\frac{\gamma_{m+1}}{|\gamma_{m+1}|} \right|= \left|-\sigma_{n_{m+1}}(\lambda_{m+1})+\frac{\sigma_{n_{m+1}} (\gamma_{m+1})}{|\gamma_{m+1}|} \right| = 1+ |\gamma_{m+1}|,
 \end{equation} and in particular $|\gamma_{m+1}| \leq \left|\lambda_{m+1}\ \right|.$ Actually, for each $j\geq m+2$, the equality $$\left\|\Delta(y) + \frac{\sigma_{n_{j}} (\gamma_{j})}{|\gamma_{j}|} \Delta (e_{n_{j}}) \right\| = \left\|y + \frac{\gamma_{j}}{|\gamma_{j}|} e_{n_{j}} \right\|$$ implies that \begin{equation}\label{eq 2808 j} \left|\lambda_{j}+\frac{\gamma_{j}}{|\gamma_{j}|} \right|= 1+ |\gamma_{j}|,
 \end{equation} and in particular $|\gamma_{j}| \leq \left|\lambda_{j}\ \right|$ for all $j\geq m+2$.\smallskip

Now, we compute \begin{equation}\label{eq 2808 distance x y} 2 |\lambda_{m+1}| = \|2 \sigma_{n_{m+1}}(\lambda_{m+1}) \Delta (e_{n_{m+1}})\| =\left\| \Delta (x) - \Delta (y)  \right\|
 \end{equation} $$ = \|x-y\| =\max \{|\mu_j-\gamma_j| :j\geq m+1\}.$$ So, there exists $j_0\geq m+1$ such that $|\mu_{j_0}-\gamma_{j_0}| = 2 |\lambda_{m+1}|.$ We deduce from the assumptions that $$2 |\lambda_{m+1}| = |\mu_{j_0}-\gamma_{j_0}| \leq |\mu_{j_0}|+|\gamma_{j_0}|\leq 2 |\lambda_{j_0}| \leq 2 |\lambda_{m+1}|,$$ which proves that $|\lambda_{j_0}| = |\mu_{j_0}|=|\gamma_{j_0}| = |\lambda_{m+1}|.$ Now, applying \eqref{eq distance 2} we get $$1+ |\lambda_{j_0}| = 1+ |\mu_{j_0}| = \left|\lambda_{j_0}+\frac{\mu_{j_0}}{|\mu_{j_0}|} \right|,$$ and thus $\mu_{j_0}=\lambda_{j_0}$. If $j_0 =m+1$ we obtain $\mu_{m+1}=\lambda_{m+1}$, as desired.\smallskip

If $j_0 \in \{m+2,\ldots,k\}$, by applying \eqref{eq 2808 j} we deduce that $\mu_{j_0}=\gamma_{j_0}$. In this case the equation in \eqref{eq 2808 distance x y} writes in the form $$2 |\lambda_{m+1}| =\left\| \Delta (x) - \Delta (y)  \right\| = \|x-y\| =\max \{|\mu_j-\gamma_j| :j\geq m+1, j\neq j_0\}.$$ Therefore, there exists $j_1\in\{m+1,\ldots, k\}$, $j_1\neq j_0$ such that $2 |\lambda_{m+1}|=|\mu_{j_1}-\gamma_{j_1}|,$ and the previous arguments show that $\mu_{j_1}=\lambda_{j_1}$ and $|\lambda_{j_1}| = |\mu_{j_1}|=|\gamma_{j_1}| = |\lambda_{m+1}|.$ If $j_1=m+1$ we have $\mu_{m+1}=\lambda_{m+1}$, otherwise it follows from \eqref{eq 2808 j} that $\mu_{j_1}=\lambda_{j_1}=\gamma_{j_1}$, and hence \eqref{eq 2808 distance x y} writes in the form $$2 |\lambda_{m+1}| =\left\| \Delta (x) - \Delta (y)  \right\| = \|x-y\| =\max \{|\mu_j-\gamma_j| :j\geq m+1, j\neq j_0,j_1\}.$$ We therefore obtain $j_2\in\{m+1,\ldots, k\}$, $j_2\neq j_0,j_1$ such that $2 |\lambda_{m+1}|=|\mu_{j_2}-\gamma_{j_2}|.$ Repeating the above arguments to $j_2$ we deduce that one of the next statement holds:
\begin{enumerate}[$(\checkmark)$] \item $j_2=m+1$ and $\mu_{m+1}=\lambda_{m+1}$;
\item There exists $j_3\in\{m+1,\ldots, k\}$, $j_3\neq j_0,j_1,j_2$ such that $2 |\lambda_{m+1}|=|\mu_{j_3}-\gamma_{j_3}|.$
\end{enumerate}
By repeating this argument a finite number of steps we derive that $\lambda_{m+1} = \mu_{m+1}$.\smallskip

Finally, the above arguments subsequently applied to $m+2,\ldots,k$ give $\mu_j = \lambda_j$ for all $j\geq m+1$.
\end{proof}

\begin{theorem}\label{t c0 satisfies the Mazur-Ulam property} Let $\Gamma$ be an infinite set. The complex space $c_0(\Gamma)$ satisfies the Mazur-Ulam property, that is, given a Banach space $X$, every surjective isometry $\Delta : S(c_0 (\Gamma))\to S(X)$ admits a unique extension to a surjective real linear isometry from $c_0(\Gamma)$ to $X$; in particular $X$ is isometrically isomorphic to $c_0(\Gamma)$.\end{theorem}

\begin{proof} Let $\Delta : S(c_0 (\Gamma))\to S(X)$ be a surjective isometry. Corollary \ref{c orthogonal have CM orthog images} assures that, for each finite subset $\Gamma_0\subset \Gamma$ the set $\{ \Delta (e_{n}) : n\in\Gamma_0 \}\subset S(X)$ is completely $M$-orthogonal. For each $n$ in $\Gamma$, let $\sigma_n : \mathbb{C}\to \mathbb{C}$ be the mapping defined by Proposition \ref{p imaginary ellinfty} and \eqref{eq Delta takes out complex or conjugate scalar at a coordinate}.\smallskip

We define a mapping $F : c_0(\Gamma) \to X$, given by $$F(x) := \sum_{n\in \Gamma} \sigma_n (x(n))\ \Delta(e_n).$$ We shall show that $F$ is well defined. For each $x\in c_0(\Gamma)$ there exists an at most countable subset $\Gamma_x$ such that $\{n\in \Gamma: x(n)\neq 0\}\subseteq \Gamma_x$ and $\displaystyle x= \sum_{n\in \Gamma_x} x(n) e_n$ and $(x(n))_{n\in \Gamma_x}$ can be regarded as a sequence in $c_0(\mathbb{N})$. Let us identify $\Gamma_x$ with $\mathbb{N}$. We claim that the sequence $\displaystyle \left(\sum_{1\leq n\in \Gamma_x}^m \sigma_n(x(n)) \Delta(e_n)\right)_{m}$ is Cauchy. Namely, given $\epsilon>0$, there exists $n_0\in \mathbb{N}$ such that $|x(n)|<\varepsilon$ for all $n\geq n_0$ in $\Gamma_x$. Pick $m_1,m_2\in \Gamma_x$ with $m_1> m_2\geq n_0$. Let $k_0\in \{m_2+1,\ldots, m_1\}$ satisfying $|x(k_0)| = \max\{|x(n)| : m_2+1\leq n \leq m_1\}$. If $|x(k_0)| =0$ we have $$\left\| \sum_{1\leq n\in \Gamma_x}^{m_1} \sigma_n(x(n)) \Delta(e_n) - \sum_{1\leq n\in \Gamma_x}^{m_2} \sigma_n(x(n)) \Delta(e_n)\right\| $$ $$= \left\| \sum_{m_2+1\leq n\in \Gamma_x}^{m_1} \sigma_n(x(n)) \Delta(e_n) \right\|= 0<\varepsilon.$$ If $|x(k_0)| >0$, then, by Proposition \ref{p algebraic ellinfty} we have $$\left\| \sum_{1\leq n\in \Gamma_x}^{m_1} \sigma_n(x(n)) \Delta(e_n) - \sum_{1\leq n\in \Gamma_x}^{m_2} \sigma_n(x(n)) \Delta(e_n)\right\| = \left\| \sum_{m_2+1\leq n\in \Gamma_x}^{m_1} \sigma_n(x(n)) \Delta(e_n) \right\|$$ $$\!=\!|x(k_0)|  \left\| \sum_{\stackrel{n\in \Gamma_x}{m_2+1\leq n}}^{m_1} \frac{\sigma_n(x(n))}{|x(k_0)|} \Delta(e_n) \right\| \!=\! |x(k_0)|  \left\| \Delta\left(\sum_{\stackrel{n\in \Gamma_x}{m_2+1\leq n}}^{m_1} \frac{x(n)}{|x(k_0)|} e_n \right) \right\|\! =\! |x(k_0)|  <\varepsilon.$$ Since $\displaystyle \left(\sum_{1\leq n\in \Gamma_x}^m \sigma_n(x(n)) \Delta(e_n)\right)_{m}$ is a Cauchy sequence and $X$ is a Banach space, the mapping $F$ is well defined, and it is clearly real linear.\smallskip

The previous arguments also show, via Proposition \ref{p algebraic ellinfty}, that for each $x\in S(c_0(\Gamma)),$ and for each natural $m$ with $1=\|x\|=\max\{|x(k)|: 1\leq k\leq m\}$, we have $$\left\| \sum_{1\leq n\in \Gamma_x}^m \sigma_n(x(n)) \Delta(e_n) \right\| =\left\| \Delta\left(\sum_{1\leq n\in \Gamma_x}^m x(n) e_n\right) \right\| =\left\| \sum_{1\leq n\in \Gamma_x}^m x(n) e_n \right\| \leq \|x\|=1.$$ Consequently, $\|F(x)\| \leq \|x\|=1,$ and thus $F$ is continuous and contractive.\smallskip

Proposition \ref{p algebraic ellinfty} implies that for every $n_1,\ldots, n_k\in \Gamma$ and every $\lambda_1,\ldots,\lambda_k\in \mathbb{C}$ with $\max\{|\lambda_1|,\ldots,|\lambda_k|\}=1$,
we have $$\Delta \left(\sum_{j=1}^{k} \lambda_j e_{n_j}\right) = \sum_{j=1}^{k} \sigma_{n_j}(\lambda_j) \Delta (e_{n_j})=F \left(\sum_{j=1}^{k} \lambda_j e_{n_j}\right).$$ Since every element in $S(c_0 (\Gamma))$ can be approximated in norm by elements in $S(c_0 (\Gamma))$ which are of the form $\displaystyle \sum_{j=1}^{k} \lambda_j e_{n_j}$ with $\max\{|\lambda_1|,\ldots,|\lambda_k|\}=1$, we deduce from the fact that $\Delta$ and $F$ are continuous that $F|_{S(c_0 (\Gamma))} = \Delta$, witnessing the desired conclusion.
\end{proof}

All technical results established above for $\mathcal{L}^{\infty}(\Gamma)$ remain valid when this space is replaced with $\ell_\infty^m$, so the above arguments in Theorem \ref{t c0 satisfies the Mazur-Ulam property} can be literally applied to obtain our last result.

\begin{theorem}\label{c finite dimensional ellinfty satisfies the Mazur-Ulam property} The finite dimensional complex space $\ell_{\infty}^{m}$ satisfies the Mazur-Ulam property, concretely, given a Banach space $X$, every surjective isometry $\Delta : S(\ell_{\infty}^{m})\to S(X)$ admits a unique extension to a surjective real linear isometry from $\ell_{\infty}^{m}$ to $X$; in particular $X$ is isometrically isomorphic to $\ell_{\infty}^{m}$.$\hfill\Box$\end{theorem}

We conjecture that the complex spaces $\ell_{\infty}(\Gamma)$ and $C(K)$ also satisfy the Mazur-Ulam property, however our current technology is not enough to prove this affirmation.

\medskip\medskip

\textbf{Acknowledgements:} Authors partially supported by the Spanish Ministry of Economy and Competitiveness (MINECO) and European Regional Development Fund project no. MTM2014-58984-P and Junta de Andaluc\'{\i}a grants FQM194 and FQM375.\smallskip

Most of the results presented in this note were obtained during a visit of A.M. Peralta at Universidad de Almer{\'i}a in July 2017. He would like to thank his coauthors and the Department of Mathematics for their hospitality during his stay.

\end{document}